\documentclass[11pt, a4paper]{amsart}

\usepackage{amsfonts}
\usepackage{amssymb}
\usepackage{amsmath, amsthm,color}
\usepackage{hyperref}
\usepackage{pdfsync}
\usepackage{bbm, dsfont}
\usepackage[margin=0.96in]{geometry}

\def\sign{\operatorname{sgn}}

\newcommand{\eps}{\varepsilon}

\newcommand{\bR}{\mathbb{R}}

\newcommand{\bE}{\mathbb{E}}

\newcommand{\ttau}{ \tau t^{\frac 1 {\gamma} }  }

\newtheorem{theorem}{Theorem}[section]

\newtheorem{lemma}[theorem]{Lemma}
\newtheorem{prop}[theorem]{Proposition}

\newtheorem{remark}[theorem]{Remark}

\numberwithin{equation}{section}

\begin{document}
\title[Poincar\'e and fractional laplacian on Hamming]{Poincar\'e type  and spectral gap inequalities with fractional Laplacians on Hamming cube}
\author[Dong Li, Alexander Volberg]{Dong Li, Alexander Volberg}
\thanks{ Volberg is partially supported by the NSF DMS-1600065.  }
%for Mathematics, and by the NSF grant DMS-1600065; 
%AV is also supported by the NSF grant DMS-1265549, and VV is also supported by the RFBR grant 14-01-00748.}
\address{Department of Mathematics, Hong Kong University of Science and Technology}
\email{madli@ust.hk \textrm{(Dong Li)}}

%\address{Department of Mathematics, University of Munich, Germany}
%\email{ \textrm{(R.\ Frank)}}

\address{Department of Mathematics, Michigan State University, East Lansing, MI 48823, USA}
\email{volberg@math.msu.edu \textrm{(A.\ Volberg)}}
\makeatletter
\@namedef{subjclassname@2010}{
  \textup{2010} Mathematics Subject Classification}
\makeatother
\subjclass[2010]{42B20, 42B35, 47A30}
\keywords{}
\begin{abstract} 
We prove here some dimension free  Poincar\'e-type inequalities on Hamming cube for functions with different spectral properties and for fractional Laplacians.
In this note the main attention is paid to estimates in $L^1$ norm on Hamming cube.
We build the examples  showing that our assumptions on spectral properties of functions cannot be dropped in general.
\end{abstract}
\maketitle

\section{Introduction}
Let $C_n:=\{-1,1\}^n$ denote  the Hamming cube, and let $x_i, i=1, \dots n$, be its coordinate functions assuming the values $\pm 1$.
If $S$ denotes a subset of $\{1,\dots, n\}$, then a monomial $x^S$ is just $x_{i_1}\cdot\dots x_{i_k}$, where $S=\{i_1, \dots, i_k\}$. If $S$ is the empty set, then we set $x^{\varnothing}=1$. 
There is a standard measure $\mu$ on $C_n$ (all points are charged by $2^{-n}$) and standard expectation with respect to this measure, it will be called $\bE$. For any $f$: $C_n \to \mathbb R$,
one can develop the expansion
\begin{align*}
f(x) = \sum_{S} a_S x^S=a_{\varnothing}+ \sum_{S\ne \varnothing} a_S x^S, 
\end{align*}
where $a_S=  \bE(f(x) x^S)=\hat f(S)$ is usually called the Fourier coefficient of $f$. The $L^2$ isometry takes
the form
\begin{align*}
\bE |f|^2 = \sum_{S} |a_S|^2.
\end{align*}

For $j\in \{1,\cdots, n\}$, define  $\nabla_j$ as
\begin{align*}
(\nabla_j f)(x)= 
\frac{ f(x_1,\cdots, x_j=1,\cdots, x_n)- f(x_1,\cdots, x_j=-1,\cdots, x_n)} 2.
\end{align*}
Then the adjoint operator $\nabla_j^{\ast}$ has the form
\begin{align*}
\nabla_j^{\ast} f = x_j \bE_j f = x_j \cdot \frac{f(x_1,\cdots, x_j=1, \cdots, x_n)
+f(x_1,\cdots, x_j=-1,\cdots, x_n)} 2.
\end{align*}
One can then introduce the Laplacian $\Delta= -\sum_{j=1}^n \nabla_j^{\ast} \nabla_j=-\sum_{j=1}^n x_j \nabla_j$. Clearly 
\begin{align*}
-\bE f \Delta g = \bE ( \sum_{j=1}^n \nabla_j f \nabla_j g).
\end{align*}

On monomials the Laplacian  acts by the rule
$$
\Delta ( x^S) = -|S| x^S,
$$
where $|S|$ denotes the cardinality of $S$, 
and thus semigroup  $e^{t\Delta}$ acts by the rule
$$
e^{t\Delta} f=a_{\varnothing}+ \sum_{S\ne \varnothing} a_S e^{-t|S|} x^S
$$
for $f= \sum_{S} a_S  x^S$.
The length of the gradient of $f$, $|\nabla f|(x)$ is defined as
$$
|\nabla f|^2(x) =\sum_{y\sim x} \left(\frac{f(x)-f(y)}{2}\right)^2,
$$
where $y\sim x$ denotes all neighbours of $x$. A point $y\in C_n$ is called a neighbour
of $x$, denoted as $y\sim x$,  if for some $i_0 \in \{1,\cdots, n\}$, we have $y_i=x_i$ for all
$i\ne i_0$, and $y_{i_0}=-x_{i_0}$.
It is easy to see that
\begin{align*}
|\nabla f|^2(x) = \sum_{j=1}^n |\nabla_j f(x)|^2,
\end{align*}
and consequently
$$
-\bE \big( f\cdot\Delta f\big) = \bE|\nabla f|^2=\sum_{S} |S| |a_S|^2\,.
$$

Then clearly, via the $L^2$ isometry mentioned earlier, 
\begin{equation}
\label{P2}
\bE |f-\bE f|^2 \le \bE|\nabla f|^2= -\bE \big( f\cdot\Delta f\big),
\end{equation}
and
\begin{equation}
\label{SpG2}
 \|e^{t\Delta} (f-\bE f)\|_2 \le e^{-t} \|f-\bE f\|_2\,.
\end{equation}

The first one is the Poincar\'e inequality in $L^2(\mu)$, the second one can be called {\it the spectral gap inequality} in $L^2(\mu)$.

\medskip

Below we are interested in such {\it dimension free} inequalities, where $L^2$ is replaced by $L^p$, especially for $p=1$ and when $\Delta$ is replaced by $\Delta_\gamma, 0<\gamma <1$,
where we define fractional Laplacian by
$$
\Delta_\gamma := - (-\Delta)^\gamma\,.
$$

\medskip

Such inequalities were studied in many situations, for us the starting point was \cite{Dong1}. The analogs on 
Hamming cube have some interesting properties and raise questions--especially about the sharp constants. 
But we do not address here the problem of sharp constants.  We wish to mention that certain estimates for fractional 
Laplacian on Hamming cube were considered in \cite{BELP}. 
Our estimates are different, but in conjunction with the estimates of \cite{BELP}, they naturally generate another set of questions which we plan to address elsewhere.

\medskip

In dealing with spectral gap estimates for  $e^{t\Delta_\gamma}$, $1<p<\infty$, we are led to the same estimate as \eqref{SpG2} with the only difference that
$e^{-t}$ gets substituted by $e^{-c_{p, \gamma} t}, c_{p, \gamma}>0$. We do not calculate $c_{p, \gamma}$ very precisely, but it is readily seen that it blows down to zero
if $p\to 1$. Moreover, we show that the inequality
\begin{equation}
\label{SpG1}
 \|e^{t\Delta_\gamma} (f-\bE f)\|_1 \le e^{-c_{1, \gamma}t} \|f-\bE f\|_1
\end{equation}
cannot generally hold with $c_{1, \gamma} >0$.

\medskip

For $\gamma=1$ this effect was carefully researched in \cite{HMO}, where for $p>1$ the constant $c_{p,1}$ is considered  in the following heat smoothing (or spectral gap)
inequality:
\begin{equation}
\label{SpGp}
 \|e^{t\Delta} (f-\bE f)\|_p \le e^{-c_{p, 1}t} \|f-\bE f\|_p\,,
\end{equation}
and it is shown that it blows down to zero when $p\to 1$. Moreover, this constant in calculated. 
\medskip

To have \eqref{SpG1} one needs something like extra assumption on the spectral properties of $f$. In the spirit of \cite{Dong1} we call $f$  {\it band limited} (or with {\it band spectrum})
if  in the Fourier decomposition $f= \sum_{S} a_S  x^S$ of $f$ one has all $a_S$ zero unless the length $|S|$ belongs to a finite set
(say, set $\{1,2, 3\}$).

\medskip

\begin{remark}
\label{band_countrex}
For band limited $f$ we prove estimate \eqref{SpG1}, but only if $\gamma <1$ ! For $\gamma =1$ there is a counterexample (see Section \ref{counterex})  to \eqref{SpG1} even for  $f$ with band spectrum.
  \end{remark}
  
  It  goes without saying that we need all constants met below to be independent of the dimension $n$ of cube $C_n$. 

\medskip

 Our spectral gap estimates are
the combination of Poincar\'e type estimates in various $L^p(\mu)$ (especially for $p=1$), hypercontractivity, and some standard convexity arguments.
The Poincar\'e estimate at $p=1$, $0<\gamma<1$, obtained below seems to be unusual. And  even Poincar\'e inequalities for $p>1$ seem to be different from the standard ones.
The next section is devoted to them.

\section{Poincar\'e-type inequalities with Laplacian}
\label{Poincare}

\begin{lemma}
Let $0<\beta \le 2$.  Let $(\Omega, d\mu)$ be a probability space.
Then for any random variable $g:\, \Omega \to \mathbb R$ with 
$\bE |g|^2<\infty$, we have
\begin{align*}
\bE |g-\bE g|^2 \ge c_1 \bE  |g|^2- 2^{\frac 1 {\beta}} \cdot | \bE \big[|g|^{\beta} \operatorname{sgn}(g) \big]|^{\frac 2 {\beta}},
\end{align*}
where $c_1>0$ is an absolute constant.
\end{lemma}
\begin{proof}
Without the loss of generality we assume $\bE |g|^2 =1$.  Let $c_1>0$ be a sufficiently small absolute constant. If
$\bE | g -\bE g |^2 \ge c_1$ we are done. Now assume $\bE |g - \bE g|^2 < c_1 \ll 1$.  Together with
the condition $\bE |g|^2 =1$ we infer that $0\le 1- |\bE g|\ll 1$. Replacing $g$ by $-g$ if necessary we 
may assume $| 1-\bE g | \ll 1$.  Let $\eta= c_1^{\frac 1{10}}$. Then for $c_1$ sufficiently
small (below the smallness of $c_1$ is independent of $\beta$ since $0<\beta\le 2$), we have
\begin{align*}
\int |g|^{\beta} \operatorname{sgn}(g) d\mu
&\ge  \int_{|g-1| \le \eta} |g|^{\beta} \operatorname{sgn}(g) d\mu -
\int_{|g-1|>\eta} |g|^{\beta} d \mu \notag \\
& \ge \sqrt{\frac 34} - \int_{|g-1|>\eta} 4\cdot (|g-1|^{\beta} +1) d\mu \notag \\
&\ge \sqrt{\frac 34} - 4 \int |g-1|^2 d\mu - 8 \int_{|g-1|>\eta} d\mu \notag \\
& \ge  \frac 1 {\sqrt 2}.
\end{align*}
  The desired inequality then obviously follows.
\end{proof}

\begin{prop} \label{prop_ng1}
Let $0<\beta\le 2$. Then for any $g:\,  \{-1, 1\}^n \to \mathbb R$, we have
\begin{align*}
\bE  |\nabla g |^2 \ge c_1 \bE |g |^2 - 2^{\frac 1 {\beta}} | \bE ( |g|^{\beta} \operatorname{sgn}(g) ) |^{\frac 2 {\beta}},
\end{align*}
where $c_1>0$ is an absolute constant.
\end{prop}

\begin{proof}
This follows from the Poincar\'e inequality with $p=2$ on Hamming cube:
\begin{align*}
\bE  |\nabla g |^2 \ge \bE |g-\bE g|^2 
\end{align*}
and the previous lemma.
\end{proof}

Next is an elementary lemma.

\begin{lemma}
\label{abp}
Let $a, b\in \bR$, $p>1$. Then there exists $\tilde c_p>0$ such that
$$
(a-b) (|a|^{p-1} \sign a- |b|^{p-1}\sign b) \ge\tilde  c_p (|a|^{\frac{p}2}\sign a - |b|^{\frac{p}2}\sign b)^2\,.
$$
Moreover,
\begin{equation}
\label{cp}
\tilde c_p = \min_{0\le t\le 1}\frac{1-t^{\frac2p}}{1-t}\cdot \frac{1-t^{\frac{2}{p'}}}{1-t} \ge  2\min \big(\frac1p, \frac1{p'}\big)\,.
\end{equation}
\end{lemma}

\begin{proof}
Notice that by symmetry we can think that either $a, b$ are both positive or that $a>0>b$.
Then by homogeneity the case $a>0>b$ is reduced to the estimate
$$
(1+x)(1+ x^{p-1}) \ge (1+ x^{\frac{p}2})^2,\, x\ge 0\,,
$$
which is the same as $2 x^{\frac{p}2} \le x + x^{p-1}$. The latter inequality is just $2AB\le A^2 + B^2$.

The case when both $a, b$ are positive becomes
$$
(1-x)(1- x^{p-1}) \ge \tilde c_p(1- x^{\frac{p}2})^2,\, 0\le x\le 1\,.
$$
Notice that this inequality is false for $p=1$, but it holds for $p>1$. This is just because after the change of variable $x=t^{\frac{2}p}$ one can observe that
$$
\lim_{t\to 1-}\frac{1- t^{\frac{2}p}}{1-t} >0,\,\, 
\lim_{t\to 1-}\frac{1- t^{\frac{2}{p'}}}{1-t} >0\,.
$$
From this one sees immediately that
$$
\tilde c_p:= \inf_{0\le x\le 1}\frac{(1-x)(1- x^{p-1}) }{(1- x^{\frac{p}2})^2} >0\,.
$$

\end{proof}

\begin{theorem} 
\label{thm1.3aa}
Let $1<p<\infty$. Then for any $f:\, \{-1,1\}^n \to \mathbb R$, we have
\begin{align*}
- \bE (\Delta f |f|^{p-1} \operatorname{sgn}(f)  )\ge C_1 \cdot c_p \cdot \bE |f|^p- 2^{\frac p2}
 \cdot c_p \cdot |\bE f|^p,
\end{align*}
where $C_1>0$ is an absolute constant, and $c_p=2\min(\frac 1p, \frac 1 {p^{\prime}})$. 
\end{theorem}

\begin{proof}
By Lemma \ref{abp} (note that we need $p>1$), we have
\begin{align}
-\bE \Delta f |f|^{p-1} \operatorname{sgn} (f) 
&=\bE( \sum_{i=1}^n \nabla_i f \nabla_i (|f|^{p-1} \operatorname{sgn}(f) ) ) \notag \\
&\ge \bE( \sum_{i=1}^n \tilde c_p |\nabla_i( |f|^{\frac p2} \operatorname{sgn}(f) ) |^2 )
\ge c_p \| \nabla ( |f|^{\frac p2} \operatorname{sgn}(f) ) \|_2^2.
\notag % \label{square}
\end{align}

Now we make a change of variable and denote $g(x) = |f(x)|^{\frac p2} \sign( f(x))$.
Note that $g$ and $f$ have the same sign. Clearly
\begin{align*}
\bE f=\bE \big[|g|^{\beta} \sign(g)\big],
\end{align*}
where $\beta= \frac 2p \in (0, 2)$ since $1< p<\infty$. The desired inequality then clearly
follows from Proposition \ref{prop_ng1}.
\end{proof}

\medskip

\section{Fractional Laplacian on Hamming cube and its spectral gap estimates}
\label{fractional}

 For $0<\gamma\le 1$, we introduce 
 $$\Delta_{\gamma}
=-(-\Delta)^{\gamma},
$$
the fractional Laplacian operator on Hamming cube via Fourier transform as
\begin{align*}
(\Delta_{\gamma} f) (x) = - \sum_{S} |S|^{\gamma} a_S x^S,
\end{align*}
for any $f=\sum_{S} a_S x^S$. In yet other words $\Delta_{\gamma}$ is simply the Fourier
multiplier $-|S|^{\gamma}$.

\medskip

The first claim of the next theorem is very well known for $p=2$. It is the claim that Laplacian on Hamming cube has a spectral gap. It is interesting that
this ``spectral gap" estimate can be extrapolated to $1<p<\infty$, and even, as we will see later, for $p=1$ sometimes.

\medskip

In Section \ref{band} we will see that with extra spectral assumptions on $f$ it holds even for $p=1$.

\begin{theorem}
\label{gap_pbigger1}
Let $1<p<\infty$. Then for any $f: \{-1,1\}^n \to \mathbb R$, we have
\begin{align*}
\| e^{t \Delta} (f -\bE f)\|_p \le e^{-k_1 t} \| f-\bE f \|_p,\quad \forall\, t>0,
\end{align*}
where $k_1=C_1 \cdot c_p$, $C_1>0$ is an absolute constant and 
$c_p = 2 \min(\frac 1p, \frac 1 {p^{\prime}})$.  Similarly for $\Delta_{\gamma}$, 
\begin{align*}
\| e^{t \Delta_{\gamma} } (f -\bE f)\|_p \le e^{-k_{\gamma} t} \| f-\bE f \|_p,\quad \forall\, t>0,
\end{align*}
where the constant $k_{\gamma} = k_1^{\gamma}$.
\end{theorem}
\begin{proof}
Without loss of generality we can assume $\bE f=0$. Denote $I(t)= \bE (|e^{t\Delta} f |^p)$.
Since $\mu$ is uniform counting measure, we can
directly differentiate 
 and this gives
\begin{align*}
\frac d {dt} I(t) &= p\bE (  \Delta g |g|^{p-1} \sign(g) ),
\end{align*}
where $g=e^{t\Delta} f$. Note that $\bE g=0$. Thus by
Theorem \ref{thm1.3aa}, we have
\begin{align*}
\frac d {dt } I(t) \le - p\cdot C_1 \cdot c_p I(t).
\end{align*}
Integrating in time then yields the desired inequality with $k_1= C_1 \cdot c_p $.  For the fractional Laplacian
case, we can use the subordination identity (see Lemma \ref{lem_subor_1})
\begin{align*}
e^{-\lambda^{\gamma}} = \int_0^{\infty} e^{-\tau \lambda} d \rho(\tau), \quad \lambda \ge 0.
\end{align*}
where $d\rho(\tau)$ is a probability measure on $[0,\infty)$.  Clearly then
\begin{align*}
e^{-t \lambda^{\gamma}} = \int_0^{\infty} e^{-\tau t^{\frac 1 {\gamma}} \lambda}
d\rho(\tau).
\end{align*}
It follows that
\begin{align*}
\| e^{t \Delta_{\gamma}} f \|_p &\le \int_0^{\infty} 
e^{-k_1 \tau t^{\frac 1{\gamma} } } d\rho(\tau) \| f \|_p \notag \\
& = e^{-k_2 t } \| f \|_p, \qquad k_2=k_1^{\gamma}.
\end{align*}
\end{proof}

\section{Band spectrum and $p=1$. The first proof}
\label{band}

We first prove a certain Poincar\'e-type inequality involving $\Delta_\gamma f, 0<\gamma<1$ in $L^1(\{-1, 1\}^n)$. It will work for functions with band spectrum.
Then we derive from it the inequality of ``spectrum gap type" for functions in $L^1(\{-1, 1\}^n)$ having  band spectrum. Namely, we get
\begin{theorem}
\label{spgap_L1_thm}
For every $\gamma\in (0,1)$ there exits $c_\gamma>0$ independent of $n$ such that for every $n$ and every $f\in L^1(\{-1, 1\}^n)$ with band spectrum (meaning that it has only, say, $1$-mode and $2$-mode only), or, more generally, finite number of modes and $\bE f=0$, we have
\begin{equation}
\label{spgap_L1}
\| e^{t\Delta_\gamma} f\|_1 \le e^{-c_\gamma t} \|f\|_1\,.
\end{equation}
\end{theorem}

\medskip

\begin{remark}
\label{Lorentz}
This result will be proved, in fact, by two different methods. The second method shows, in particular, that the $L^1$-norm can be changed to any shift invariant norm (as $\{-1,1\}^n$ is isomorphic to 
$\mathbb F_2^n$ and shift can be understood on this group). In particular, one gets the spectral gap theorem on any Lorentz or Orlicz space on cube $C_n$.
\end{remark}

\begin{remark}
\label{gammEqualTo1}
This result is false if $\gamma=1$ even for band limited $f$. The counterexample is in Section \ref{counterex}.
\end{remark}

\medskip

However, the Poincar\'e inequality in $L^1(\{-1, 1\}^n)$ from the  subsection \ref{PoincareL1band} below seems to have an independent interest and it looks slightly unusual.

\medskip

But first we need a known result on hypercontractivity.

\subsection{Hypercontractivity helps}
\label{hyper}

\begin{theorem}
\label{locDelta_th}
Let $f$ be Fourier localized to finite number of (say $k$) modes with $\bE f =0$. 
Then 
\begin{align*}
\| e^{t\Delta} f \|_1 \le e^{-\frac 12 t} \|f\|_1, \quad t \ge 3k \log 3.
\end{align*}
\end{theorem}

\begin{proof}
This follows easily from Theorem 9.22 of \cite{OD}. We will repeat the reasoning for the sake of convenience of the reader. By using the Bonami lemma (see pp. 247 of \cite{OD}), we have
\begin{align*}
\|f \|_4 \le 3^{\frac k2} \|f\|_2 \le 3^{\frac k2} \|f \|_1^{\frac 13} \|f \|_4^{\frac 23}.
\end{align*}
This implies $\|f\|_4 \le 3^{\frac 32 k} \|f\|_1$. Thus
\begin{align*}
\| e^{t\Delta} f \|_1 \le \| e^{t\Delta} f\|_2 \le e^{-t} \|f\|_2
\le e^{-t} \|f\|_4 \le e^{-t} 3^{\frac 32 k} \| f\|_1.
\end{align*}
\end{proof}

\begin{remark}
\label{p1alpha1tsmall}
The inequality
\begin{equation}
\label{locDelta_small}
\|e^{t\Delta} f \|_1 \le e^{-ct} \|f\|_1
\end{equation}
is not  true for small $t$ even for band limited $f$. The counterexample in subsection \ref{conterexSmall_t} shows that.
\end{remark}

\subsection{Poincar\'e inequality with $\Delta_\gamma$ in $L^1$. The first proof}
\label{PoincareL1band}

Recall that $\Delta_\gamma= -(-\Delta)^\gamma$.

\begin{theorem}
\label{Poincare_L1_thm}
For every $\gamma\in (0,1)$ there exits $b_\gamma>0$ independent of $n$ such that for every $n$ and every $f\in L^1(\{-1, 1\}^n)$ with finite number of Fourier $k$ modes  (i.e. localized
to $1$-mode, $\cdots$, $k$-mode) and $\bE f=0$, we have
\begin{equation}
\label{Poincare_L1}
b_\gamma \cdot \alpha_k \|f\|_1 \le \bE [(-\Delta_\gamma f)\cdot \sign f
\cdot {\bf 1}_{f\ne 0}] -\bE [ |\Delta_\gamma f|\cdot {\bf 1}_{f=0}]\,,
\end{equation}
where $\alpha_k= k^{-\gamma}\cdot 3^{-3k}$.
\end{theorem}

\begin{proof}
Let $\gamma\in (0,1)$, put 
$$
C_\gamma:= \int_0^\infty (1-e^{-u})\frac{du}{u^{1+\gamma}} <\infty\,.
$$
It is then obvious that for any function $f$ on the cube such that $\bE f=0$ one has
\begin{align*}
-\Delta_\gamma f = C_\gamma^{-1} \int_0^\infty \big( f - e^{t\Delta} f\big)  \frac{dt}{t^{1+\gamma}}.
\end{align*}
Note that here and below convergence is not an issue since we are on the Hamming cube.

Now clearly
\begin{align*}
&C_{\gamma} \Bigl( \bE [(-\Delta_\gamma f)\cdot \sign f
\cdot {\bf 1}_{f\ne 0}] -\bE [ |\Delta_\gamma f|\cdot {\bf 1}_{f=0}]  \Bigr)\notag \\
=&\bE (  \int_0^\infty \big(| f| - \operatorname{sgn}(f)
\cdot {\bf 1}_{f\ne 0} e^{t\Delta} f\big)  \frac{dt}{t^{1+\gamma}} )
-\bE(  |\int_0^\infty \big(  e^{t\Delta} f\big)  \frac{dt}{t^{1+\gamma}}| \cdot {\bf 1}_{f=0}) \notag \\
\ge & \int_0^{\infty} (\|f \|_1 - \|e^{t\Delta} f\|_1)\cdot \frac {dt}{t^{1+\gamma}}.
\end{align*}
Since $\| e^{t\Delta} f\|_1 \le \|f\|_1$ for each $t\ge 0$. We can restrict the integral to
$3k\log 3 \le t \le 6k \log 3$ and then apply  Thorem \ref{locDelta_th}.

\end{proof}

\subsection{The first proof of Theorem \ref{spgap_L1_thm} via Poincar\'e inequality in $L^1$}
\label{FirstPr}

Denote
$$
I(t) = \bE \big| e^{t\Delta_\gamma} f|\,.
$$
We want to estimate $\frac{d}{dt} I(t)$ for a test function $f$.
Let $F:=F_t:= e^{t\Delta_\gamma} f$.  Then for $\eps>0$, we have
\begin{equation}
\label{eps}
|e^{-\eps (-\Delta_\gamma) } F| - |F|=\begin{cases}
 \eps \sign F \cdot (\Delta_\gamma F)+ O(\eps^2),\,\, \text{if}\,\, F(x)\neq 0;\\
\eps |\Delta_\gamma F| + O(\eps^2),\,\, \text{if}\,\, F(x)= 0\,.
\end{cases}
\end{equation}
 %Let us think that $f$ is a test function with only finitely many non-zero Fourier--Walsh coefficients.
One should keep in mind that we are on the discrete Hamming cube and as such interchanging
integrals with differentiation should not be an issue.
Now if we look at   $\frac{d}{dt} I(t)$ as the expression
$$
\frac{d}{dt} I(t) := \lim_{\eps\to 0} \frac{I(t+\eps) -I(t)}{\eps},
$$
we notice that the limit exists and  that we can go to the limit under the sign of $\bE$.
So we get from \eqref{eps} that
$$
\frac{d}{dt} I(t) = \bE \big( \sign F_t\cdot (-\Delta_\gamma F_t) \cdot {\bf 1}_{F_t\neq 0}\big) - \bE\big( |\Delta_\gamma F_t|\cdot {\bf 1}_{F_t=0}\big) \le - \tilde b_\gamma \bE |F_t| \,.
$$
The last inequality follows from Theorem \ref{Poincare_L1_thm}. Hence,
$$
 \frac{d}{dt} I(t) \le - \tilde b_\gamma I(t),\,\, I(0)= \|f\|_1\,.
 $$
 Therefore,  \eqref{spgap_L1} is proved for test functions $f$ with universal constant, and so Theorem \ref{spgap_L1_thm} is proved.

\section{The second proof of Theorem \ref{spgap_L1_thm} via the modification of the kernel of $e^{t\Delta_\gamma}$} 
\label{SecondPr}

We begin with a well-known fact connected with the subordination of fractional heat operators. We need
some sharp asymptotics  which will play some role in the perturbation argument later. For the sake
of completeness we include the proof (even for some well-known facts).

\begin{lemma} \label{lem_subor_1}
Let $0<\gamma<1$. Then
\begin{align*}
e^{-\lambda^{\gamma}}
= \int_0^{\infty} e^{-\lambda \tau} p_{\gamma}(\tau) d\tau, \quad \lambda\ge 0,
\end{align*}
where $p_{\gamma}$ is a probability density function on $\mathbb R$ satisfying:
\begin{itemize}
\item $p_\gamma $ is infinitely differentiable with bounded derivatives of all orders,  and $p_{\gamma}(\tau)=0$ for any $\tau \le 0$.
\item $\lim_{\tau \to \infty} \tau^{1+\gamma} p_{\gamma}(\tau)= C_{\gamma}$,
where
\begin{align*}
\frac 1 {C_{\gamma}} = \int_0^{\infty} \frac {1-e^{-\tau}} {\tau^{1+\gamma}} d\tau.
\end{align*}
\end{itemize}
\end{lemma}
\begin{remark}
For $\gamma=1/2$ it is well-known that $p_{\frac 12}$ admits
an explicit representation. One can just observe
\begin{align*}
e^{-|x|} &= \frac 1 {\pi} \int_{\mathbb R}  \frac 1 {1+\xi^2} e^{i \xi \cdot x } d\xi \notag \\
&=\frac 1 {\pi} \int_{\mathbb R} \int_0^{\infty}
e^{- t(1+\xi^2)} dt e^{i \xi \cdot x } d\xi \notag \\
& = \frac 1 {\sqrt {\pi}} \int_0^{\infty}
t^{-\frac 12} e^{- \frac {x^2} {4t} } e^{-t}dt= \frac 1 {2\sqrt{\pi}}
\int_0^{\infty} e^{-\tau x^2} \tau^{-\frac 32} e^{-\frac 1 {4\tau}}d\tau.
\end{align*} 
This is essentially the same formula as in \cite{ChaLem} after a change of variable.
\end{remark}
\begin{proof}
For simplicity we shall write $p_{\gamma}$ as $p$. We first show its existence.
For any $z= r e^{i\theta}$ with $\theta \in [-\frac {\pi}2, \frac {\pi}2]$, we fix the branch
of $z^{\gamma}$ such that $z^{\gamma}=r^{\gamma} e^{i\gamma \theta}$. Define
for $x>0$:
\begin{align*}
p(\tau)= \frac 1 {2\pi i} \int_{x-i\infty}^{x+i\infty} e^{-z^{\gamma}} e^{\tau z} dz.
\end{align*}
By deforming the contour it is easy to check that the integral is independent of $x$. By using  a large
semi-circle to the right one can verify that $p(\tau)$ vanishes for $\tau <0$. Furthermore 
one can take the limit $x \to 0+$ to obtain
\begin{align} \label{eq_ptau1}
p(\tau)= \frac 1 {\pi}
\int_0^{\infty} e^{-y^{\gamma} \cos(\frac {\gamma \pi} 2)}
\cos ( \tau y - y^{\gamma } \sin (\frac {\gamma \pi} 2) ) dy.
\end{align}
From this it is evident that $p$ has bounded derivatives of all orders. For $z \in \{x+iy: \, x>0, \, y\in \mathbb R\}$,
we have
\begin{align*}
e^{- z^{\gamma}} = \int_0^{\infty} e^{- \tau z} p(\tau) d\tau.
\end{align*}
In particular this identity holds for any $z=\lambda>0$. Now to show $p \ge 0$ one can just appeal
to the Bernstein theory.  More directly one can just use the fact  that
\begin{align*}
\lim_{\lambda \to \infty} e^{-\lambda \tau} \sum_{k \le \lambda x} 
\frac { (\lambda \tau)^k } {k!} =\begin{cases}
1, \quad \text{if $0\le \tau \le x$;} \\
0, \quad \text{otherwise.}
\end{cases}
\end{align*}
Since $e^{-\lambda^{\gamma}} $ is completely monotone, one can then deduce 
\begin{align*}
\int_{x_1}^{x_2} p(\tau) d\tau \ge 0, \quad \text{for any $0\le x_1<x_2<\infty$}.
\end{align*}
This then yields $p\ge 0$.  Finally to show the asymptotic of $p$, we use \eqref{eq_ptau1} and
partial integration to write
\begin{align*}
\pi \tau p(\tau)
=\gamma \operatorname{Re}
\Bigl [
\frac {z_0} i 
\int_0^{\infty} e^{- y^{\gamma} z_0} e^{i\tau y} y^{\gamma-1} dy \Bigr],
\end{align*}
where $z_0= e^{i\gamma\pi/2}$. By a further change of variable, we get
\begin{align*}
\pi \tau^{1+\gamma} p(\tau)
= \gamma
\operatorname{Re}
\Bigl[
\frac {z_0}i
\int_0^{\infty} e^{-\frac{y^{\gamma} z_0} {\tau}}
e^{iy} y^{\gamma-1} dy\Bigr].
\end{align*}
Now one can deform the contour integral inside the square bracket slightly to a straight line
making a very small angle with the positive real axis. This then easily yields the existence of
the limit as $\tau \to \infty$. To calibrate this constant, one can use the simple relation
\begin{align*}
1-e^{-R^{-\gamma}} = \int_0^{\infty} (1-e^{-\frac {\tau}R} ) p(\tau) d\tau
=\int_0^{\infty} (1- e^{-\tau}) Rp(\tau R) d\tau.
\end{align*}
This gives
\begin{align*}
R^{\gamma} (1- e^{-R^{-\gamma}})
= \int_0^{\infty}  \frac {1-e^{-\tau} } {\tau^{1+\gamma}}
p(\tau R)  \cdot (\tau R)^{1+\gamma} d\tau.
\end{align*}
Sending $R\to\infty$ then yields the constant.
\end{proof}

Our next lemma is the heart of the matter. It shows that a carefully chosen 
perturbation of the fractional heat kernel can leave invariant the
``band-limited" portion and decrease the $L^1$ operator norm.
Compared with the continuous setting in \cite{Dong1} the discrete Hamming
cube case requires a new and nontrivial twist.

\begin{lemma} \label{lem5.3a}
Let $0<\gamma<1$. There exists $t_0=t_0(\gamma) \in (0,1)$ such that
the following hold. Consider the kernel $K_t^{\gamma}$ corresponding to
$e^{t\Delta_{\gamma}}$:
\begin{align*}
K_t^{\gamma}(x) = \sum_{S\subset [n]} e^{-t|S|^{\gamma}} x^S.
\end{align*}
there exists a modification of $K_t^{\gamma}$ which we denoted as 
$\tilde K_t^{\gamma}$, such that:
\begin{itemize}
\item $\tilde K_t^{\gamma}$ is still non-negative, and 
$\| \tilde K_t^{\gamma} (\cdot )\|_{L_x^1} \le e^{-c_0 t}$, for all $0<t\le t_0$, where $c_0>0$
depends only on $\gamma$;
\item $\widehat{\tilde K_t^{\gamma} } (S)= 
\widehat{K_t^{\gamma}}(S)$, for any $S\subset [n]$ with $1\le |S|\le 2$.
\end{itemize}
\end{lemma}

\begin{proof}
By using Lemma \ref{lem_subor_1} we can write
\begin{align*}
K_t^{\gamma}(x)
=\int_0^{\infty} \prod_{j=1}^n (1+ e^{-\tau t^{\frac 1 {\gamma}}} x_j)
p_{\gamma}(\tau) d\tau.
\end{align*}
By Lemma \ref{lem_subor_1}, we may choose $R_0=R_0(\gamma)>10$ sufficiently large such that
\begin{align}
\label{p_size}
p_{\gamma}(\tau) \ge \frac 12 \tau^{-(1+\gamma)} C_{\gamma}, 
\quad\forall\, \tau \ge R_0.
\end{align}
Now define $t_0=  R_0^{-\gamma}$.  
Let us assume $S=\{1,2\}$. Any other $S$ with $|S|=2$ will be treated in the same way. Moreover, as the reader will
see
any finite $|S|$ can be treated in exactly same way.
For $0<t\le t_0$, we construct the modified kernel function
as
\begin{align*}
\tilde K_t^{\gamma}(x)
= \int_0^{\infty} \prod_{j=1}^n (1+e^{-\tau t^{\frac 1 {\gamma} }} x_j)
(p_{\gamma}(\tau)- \kappa t^{\frac {1+\gamma}{\gamma}} 
\varphi(t^{\frac 1{\gamma}} \tau) ) d\tau,
\end{align*}
where $\kappa>0$ is a sufficiently small constant, and $\varphi$ is a bump function
supported in $[1,2]$ satisfying:
\begin{align*}
&\int_0^{\infty} e^{-\tau} \varphi(\tau) d\tau =0, \quad
\\
&\int_0^{\infty} e^{-2\tau} \varphi(\tau) d\tau=0, \quad \int_0^{\infty} \varphi(\tau) d\tau>0.
\end{align*}
Note that the first two equalities easy imply that $\widehat{\tilde K_t^{\gamma} } (S)= 
\widehat{K_t^{\gamma}}(S)$, for our $S=\{1,2\}$. 

\medskip

 On the other hand,
on the support of $\varphi$, we have $\tau \sim t^{-\frac1{\gamma}}$, and one can easily choose (by using \eqref{p_size})
$\kappa$ sufficiently small such that 
\begin{align*}
p_{\gamma}(\tau)- \kappa t^{\frac {1+\gamma}{\gamma}} 
\varphi(t^{\frac 1{\gamma}} \tau) \ge 0.
\end{align*}
Since $\tilde K_t^{\gamma}$ is non-negative, we clearly have
\begin{align*}
\| \tilde K_t^{\gamma} \|_{L_x^1} = \text{$0$-mode of $\tilde K_t^{\gamma}$}
=1- \kappa t \int_0^{\infty} \varphi(\tau) d\tau.
\end{align*}

\medskip

Notice that any $S$ with fixed finite $|S|$ can be treated by the same approach. For example, if $S$ were $\{2, 3, 7\}$ we would replace the previous 
requirements by the following ones:
\begin{align*}
&\int_0^{\infty} e^{-2\tau} \varphi(\tau) d\tau =0, \quad
\\
&\int_0^{\infty} e^{-3\tau} \varphi(\tau) d\tau=0, 
\\
&\int_0^{\infty} e^{-7\tau} \varphi(\tau) d\tau=0, 
\quad \int_0^{\infty} \varphi(\tau) d\tau>0.
\end{align*}
\end{proof}

\medskip

\begin{remark}
Of course $\kappa$ depends on $|S|$, and even on $S$ itself. But if one fixes the ``band"  $S$ and starts to increase the dimension $n$, 
this constant $\kappa$ will not be depending on $n$. We choose  function $\varphi$ with orthogonality conditions as above that depend on $S$ but have nothing to do with $n$.
It would be interesting to measure the dependence on $S$.
\end{remark}

We have the following general inequality for band localized functions $f$ for some universal $c_\gamma>0$ (if $0<\gamma <1$).
\begin{theorem}
\label{band_peq1}
Let $0<\gamma<1$.
Let the function $f: C_n\to \bR$ is band localised to, say, the first and the second mode only, then independent of $n$ and for all such $f$ we have
\begin{equation}
\label{sub-alpha1}
\|e^{t \Delta_{\gamma}} f\|_p \le e^{-c_\gamma t} \|f\|_p, \,\,\forall\, t\ge 0,\, 1\le p\le \infty,
\end{equation}
where $c_{\gamma}>0$ depends only on $\gamma$.
Moreover, the norm $\|\cdot\|_p, 1\le p\le \infty$ can be replaced here by the norm of any shift invariant Banach space on Hamming cube.
\end{theorem}
\begin{proof}
Since $e^{t \Delta_{\gamma}} f$ is still band localised, it suffices to prove the result
for $0<t<t_0$ with $t_0=t_0(\gamma)$ small.  This follows directly from Lemma
\ref{lem5.3a} and Young's inequality.
\end{proof}

\bigskip

\begin{remark}
For $p>1$ and $\gamma=1$ we have even stronger Theorem \ref{gap_pbigger1}. It is stronger because it can be formulated  as 
\begin{equation}
\label{sub-alphaeq1}
\|e^{t \Delta} f\|_p \le e^{-c_1 t} \|f\|_p, \,\, t\ge 0,\, 1< p<\infty,
\end{equation}
independently of $n$ for all functions $f$ that are very weakly spectral localized, namely, for $f$ such that only $0$-mode vanishes: $\bE f=0$.
\end{remark}
\begin{remark}
Also for $p=1, \gamma=1$ one has the estimate \eqref{sub-alpha1}---but only for large $t$, see
Theorem \ref{locDelta_th}. As to the case $p=1, \gamma=1$, $t$ is small, and $f$ is band localized, subsection \ref{conterexSmall_t} shows that such drop of norm can be false.
So this is the case when even for band localized functions we do not have ``spectral gap" type inequality.  But as soon as either 1) $p>1$ and any $\gamma\le 1$ or  2) $\gamma<1, p=1$ we have ``spectral gap" inequality
$$
\| e^{t\Delta_\gamma} f\|_p \le e^{-c t} \|f\|_p, \, c>0\,.
$$
In case 1)  we just need very weak spectral localization, namely, just $\bE f=0$. In  case 2) we used that $f$ is band localized.  This condition cannot be dropped as counterexample is subsection \ref{gamma_less1_counterex} shows.
\end{remark}

\section{Counterexamples}
\label{counterex}

\subsection{Counterexample to $\|e^{t\Delta}f\|_1 \le e^{-ct}\|f\|_1$, $\bE f=0$}
\label{counterexNotBand}

One cannot get independent of $n$ estimate of Theorem \ref{gap_pbigger1} for $p=1$.
In fact, let $f(1,\dots, 1) = 2^{n-1}$, $f(-1,\dots, -1) = -2^{n-1}$, and $f(x)=0$ for all other points $x\in \{-1,1\}^n$.
Then $\bE f=0, \|f\|_1=1$, and 
$$
e^{t\Delta} f (x) = 2^{-1}\big(\prod_{i=1}^n (1+ e^{-t} x_i) - \prod_{i=1}^n (1- e^{-t} x_i)\big)\,.
$$
Hence,
\begin{align}
\label{L1}
&\|e^{t\Delta} f\|_1 = \frac{1}{2^{n+1}}\sum_{k=0}^n\binom{n}{k}\big | (1+e^{-t})^{n-k}(1-e^{-t})^k - (1-e^{-t})^{n-k}(1+e^{-t})^k \big| =\nonumber
\\
&\frac{1}{2^{n}}\sum_{0\le k\le \frac{n}2} \binom{n}{k} \big ( (1+e^{-t})^{n-k}(1-e^{-t})^k - (1-e^{-t})^{n-k}(1+e^{-t})^k \big)\,.
\end{align}

\medskip

Now let us assume that there exists a universal constant $\kappa<1$ such that for all $n$ and all functions $f\in L^1(\{-1, 1\}^n), \bE f=0$, there exists $t_0$ such that  for
all $t\ge t_0$
\begin{equation}
\label{drop1}
\|e^{t\Delta} f\|_1 \le \kappa \|f\|_1, \,\, \text{if}\,\, \bE f=0\,.
\end{equation}

Then by semi-group property \eqref{drop1}  would imply the universal $t_1 =2 t_0 \cdot \log 2/ \log\frac1{\kappa}$ such that for all $n$ simultaneously
\begin{equation}
\label{drop2}
\|e^{t_1\Delta} f\|_1 < \frac12 \|f\|_1, \,\, \text{if}\,\, \bE f=0\,.
\end{equation}

\begin{prop}
\label{almost1}
Let $0<\epsilon\le 1/2$. Then for $n$ sufficiently large, we have
\begin{align*}
\frac 1 {2^{n}} \sum_{0\le k \le n/2}
\binom{n}{k}
\cdot (  (1+\epsilon)^{n-k} (1-\epsilon)^k - (1+\epsilon)^k (1-\epsilon)^{n-k})
\ge \frac12(1- (1-\epsilon^2)^{\frac n2}).
\end{align*}
\end{prop}

\begin{proof}
We have
\begin{align*}
2\cdot\operatorname{LHS}
&\ge \sum_{0\le k\le \frac n2} \frac 1 {2^n}
\binom{n}{k} \cdot (1+\epsilon)^{n-k} (1-\epsilon)^k
+\sum_{k>n/2}  \frac 1 {2^n}
\binom{n}{k} \cdot (1+\epsilon)^{n-k} (1-\epsilon)^k \notag \\
& \quad - \sum_{0\le k\le \frac n2} \frac 1 {2^{n-1}} 
\binom{n}{k} \cdot (1+\epsilon)^{\frac n2} (1-\epsilon)^{\frac n2} \notag \\
& \ge 1- (1-\epsilon^2)^{\frac n2},
\end{align*}
where in the last inequality we may assume $n$ is an odd integer so that $k=n/2$ cannot
be obtained. If $n$ is even, one can get a similar bound.
\end{proof}
Now we use \eqref{L1} and the Proposition to come to contradiction with \eqref{drop2}. Hence \eqref{drop1} is false too.

\subsection{Counterexample to $\| e^{t\Delta_{\gamma}} f \|_{L^1}
\le e^{-ct} \| f \|_{L^1}$ for $f$ with $\mathbb E f=0$}
\label{gamma_less1_counterex}

Fix $0<\gamma<1$.  Again we shall argue by contradiction. Assume the desired estimate is true. Similar to
the Laplacian case, this would imply that there exists
universal $t_1>0$ independent of $n$, such that for all $f$ with $\mathbb E f=0$, 
we have
\begin{align*}
\| e^{t_1\Delta_{\gamma} } f \|_1 \le \frac 14 \| f\|_1.
\end{align*}
Now take the same $f$ as in the Laplacian case. By using the subordination formula
\begin{align*}
e^{-t\lambda^{\gamma}} = \int_0^{\infty} e^{-\tau t^{\frac 1 {\gamma}} \lambda}
d\rho(\tau), 
\end{align*}
we get
\begin{align*}
(e^{t\Delta_{\gamma} } f )(x)
= \frac 12 \int_0^{\infty}
( \prod_{j=1}^n ( 1+e^{ -\tau t^{\frac 1{\gamma}} } x_j)
-\prod_{j=1}^n(1-e^{-\tau t^{\frac 1 {\gamma}} } x_j) ) d\rho(\tau).
\end{align*}
Hence
\begin{align*}
&\| e^{t\Delta_{\gamma}} f \|_1 \notag \\
=&\frac{1}{2^{n}}\sum_{0\le k\le \frac{n}2} \binom{n}{k} \int_0^{\infty} \big 
( (1+e^{-\ttau})^{n-k}(1-e^{-\ttau})^k - (1-e^{-\ttau})^{n-k}(1+e^{-\ttau})^k \big)
d\rho(\tau) \notag \\
\ge & \frac 12 \int_0^{\infty}
(1- (1-e^{-2\ttau})^{\frac n2} ) d\rho(\tau).
\end{align*}
Now take $t=t_1$ and send $n$ to infinity. We clearly arrive at a contradiction!

\subsection{Counterexample to $\|e^{t\Delta} f\|_1\le e^{-ct} \|f\|_1$
for band-limited $f$ with small $t$}
\label{conterexSmall_t}

Consider the Gaussian space case. Let $\rho(x)= e^{-\frac {x^2} 2}$ and consider
$f(x)=x^3= \operatorname{He}_3 (x)+ 3 \operatorname{He}_1 (x)$. 
Denote $\Delta_{\operatorname{ou}} f = f^{\prime\prime} - x f^{\prime}$.  Then one
can verify that 
\begin{align*}
\int_{f \ne 0} (-\Delta_{\operatorname{ou}} f ) \operatorname{sgn}(f)  \rho(x) dx =0.
\end{align*}
This in turn implies that
\begin{align*}
\| e^{t\Delta} f \|_1 \ge \|f\|_1 - O(t^2),
\end{align*}
for small $t$, which of course contradicts $\|e^{t\Delta} f\|_1 \le e^{-c_0t} \|f\|_1 \le (1-c_0t +O(t^2))\|f\|_1$, $c_0>0$.

\medskip

%%%%%%%%%%%%%%%%
%%%%%%%%%%%%%%%%%%%%%%
%%%%%%%%%%%%%%%%%%%%%%%%%%%

\end{document}